\documentclass{amsart}
\usepackage{amssymb,amsthm,amsmath}
\usepackage{enumerate}
\usepackage{cite}
\usepackage[hidelinks]{hyperref}
\usepackage{dsfont}

\def \Nats {\mathds{N}}
\def \Ints {\mathds{Z}}

\def \restrictedto {\upharpoonright}

\DeclareMathOperator{\cl}{cl}

\newcommand{\otn}[2]{#1_{1},\dots ,#1_{#2}}

\newcommand{\set}[1]{\{#1\}}
\newcommand{\setarg}[2]{{\{#1\ |\ #2\}}}
\newcommand{\setcolon}[2]{\{#1 : #2\}}

\def \strong {\leqslant}

\DeclareMathOperator{\geometry}{G}

\newtheorem{lemma}{Lemma}[section]
\newtheorem{theorem}[lemma]{Theorem}

\newtheorem{prop}[lemma]{Proposition}

\theoremstyle{definition}
\newtheorem{definition}[lemma]{Definition}
\newtheorem{fact}[lemma]{Fact}
\newtheorem{observation}[lemma]{Observation}
\newtheorem*{notation*}{Notation}
\newtheorem{remark}[lemma]{Remark}

\def \dm {d}

\DeclareMathOperator{\bigM}{M}
\DeclareMathOperator{\clq}{clq}
\DeclareMathOperator{\geo}{geo}
\DeclareMathOperator{\sym}{sym}
\newcommand{\pg}[1]{\mathcal{#1}}

\title{An ab initio construction of a geometry}
\author{Omer Mermelstein}
\date{\today}
\address{Department of Mathematics, Ben Gurion University of the Negev \\ P.O.B 653, Be'er Sheva 8410501, Israel}
\email{omermerm@math.bgu.ac.il}
\subjclass{Primary 03C30,03C45; Secondary 03C13}
\keywords{Hrushovski construction, Predimension}
\begin{document}
\begin{abstract}
We show that the geometry of Hrushovski's ab initio construction for a single $n$-ary relation not-permitting dependent sets of size less than $n$, when restricted to $n$-tuples, can be itself constructed as a Hrushovski construction.
\end{abstract}
\maketitle

\section{Introduction}
In the early 1990s Hrushovski \cite{Hns,HrushovskiSecond} constructed his celebrated counter examples to Zilber's Conjecture -- a suggested classification of strongly minimal sets according to their (forking) geometry. The construction of \cite{Hns} and its variants, where this method is applied to the class of finite hypergraphs with hereditarily less edges than vertices, were dubbed the Hrushovski \emph{ab initio} constructions. The geometries associated to the ab initio construction are all \emph{flat}, a property prohibiting the any algebraic structure.

In \cite{EvansMatroid} Evans surveys the connections between Matroids and Hrushovski constructions. Works of Evans and Ferreira \cite{DavidMarcoOne,DavidMarcoTwo} explore how variation in the amalgamation class affect the geometry (matroid) associated to the resulting ab initio construction. The paper \cite{DavidMarcoOne} closes with the question of whether the geometry associated to a Hrushovski construction is a generic construction itself.

This paper answers Evans and Ferreira's question in the positive, for a subclass of the Hrushovski constructions -- the constructions given by a single $n$-ary relation and not admitting dependent sets of size less than $n$. Moreover, we show that, the geometry is essentially, like the ab initio construction itself, a generic construction given by a predimension function. We achieve this by using the methods of \cite{Mermelstein2016} in conjunction with the predimension function proposed in \cite{HassonMermelstein2017}.

\section{Preliminaries}
\subsection{Generic structures}
A category whose objects form a class of finite (relational) structures $\mathbb{A}$, closed under isomorphisms and substructures, and whose morphisms, $\strong$, are (not necessarily all) embeddings, is an \emph{amalgamation class} (or has the \emph{Amalgamation property} and \emph{Joint Embedding property}) if:
\begin{itemize}
\item [(AP)]
If $A,B_1,B_2\in \mathbb{A}$ are such that $A\strong B_1,B_2$, then there exists some $D\in\mathbb{A}$ and embeddings $f_i:B_i \to D$ such that $f_i[B_i]\strong D$ and $f_1\restrictedto A = f_2\restrictedto A$.	
\item [(JEP)]
If $A_1,A_2\in \mathbb{A}$, then there exists some $B\in\mathbb{A}$ and embeddings $f_i:A_i\to B$ such that $f_i[A_i]\strong B$ for $i=1,2$.
\end{itemize}
By Fra\"iss\'e's Theorem, to every amalgamation class is associated a unique (up to isomorphism) countable structure $\mathbb{M}$ satisfying
\begin{enumerate}
\item
Every finite substructure of $\mathbb{M}$ is an element of $\mathbb{A}$.
\item
Whenever $A\strong \mathbb{M}$ and $A\strong D\in\mathbb{A}$, there is an embedding $f:D\to \mathbb{M}$ fixing $A$ pointwise such that $f[D]\strong\mathbb{M}$.
\end{enumerate}
We call $\mathbb{M}$ a generic structure for $\mathbb{A}$.

\subsection{Geometries}
We interchangeably consider a geometry (a finitary matroid whose every element is closed) on a set $X$ as:
\begin{enumerate}
\item
A closure operator $\cl:P(X)\to P(X)$;
\item
A dimension function $d:P(X)\to \Nats\cup{\infty}$;
\item
A first order structure with relations $\setarg{I_n}{n\in \Nats}$ where $I_n\subseteq X^n$ is the set of independent $n$-tuples.
\item
A first order structure with relations $\setarg{D_n}{n\in \Nats}$ where $D_n\subseteq X^n$ is the set of dependent $n$-tuples.
\end{enumerate}

\begin{notation*}
For a geometry $\pg{G}$, we denote its closure operator $\cl_{\pg{G}}$ and its dimension function $\dm_{\pg{G}}$.
\end{notation*}

\begin{definition}
Say that a geometry $\pg{G}$ is \emph{flat} if whenever $E_1,\dots,E_n$ are closed in $\pg{G}$, then
\[
\sum_{s\subseteq \set{1,\dots,k}} (-1)^{|s|}d(E_s) \leq 0 
\]
where $E_{\emptyset}=\cl(\bigcup_{i=1}^n E_n)$ and $E_s = \bigcap_{i\in s}E_i$ for every $\emptyset\neq s\subseteq \set{1,\dots,n}$.
\end{definition}

\begin{definition}
We say that a geometry $\pg{G}$ is $n$-pure if $n$ is the maximal natural number such that every $n$-tuple of $\pg{G}$ is independent.
\end{definition}

\subsection{Clique predimension}

\newcommand{\lang}[1]{\mathcal{L}_{#1}}
\newcommand{\maxcliques}[1]{\bigM(#1)}
\newcommand{\Cclq}{\mathcal{C}^{\clq}}
\newcommand{\Mclq}{\mathbb{M}^{\clq}}
\newcommand{\cardstar}[1]{|#1|_*}
\newcommand{\predim}{\delta_s}

\newcommand{\Cgeo}{\mathcal{C}^{\geo}}
\newcommand{\Mgeo}{\mathbb{M}^{\geo}}
\newcommand{\Csym}{\mathcal{C}^{\sym}}
\newcommand{\Msym}{\mathbb{M}^{\sym}}

Fix a symmetric irreflexive $n$-ary relation $S$.

\begin{definition}
Let $A$ be some finite $\set{S}$-structure. We say that $K\subseteq A$ with $|K|\geq n$ is a \emph{clique} in $A$ if $[K]^n \subseteq S^A$. We say that $K$ is a \emph{maximal} clique in $A$ if there is no clique $K'\subseteq A$ such that $K'\supset K$. Define $\maxcliques{A}$ to be the set of maximal cliques of $A$.
\end{definition}

\begin{notation*}
For a finite set $X$ define $\cardstar{X} = \max\set{0,|X|-(n-1)}$
\end{notation*}

\begin{definition}
For every finite $\set{S}$-structure $A$ define
\begin{gather*}
s(A) = \sum_{K\in\maxcliques{A}} \cardstar{K}
\\
\predim(A) = |A| - s(A)
\end{gather*}
\end{definition}

For a finite $\set{S}$-structure $A$ and a finite substructure $B\subseteq A$ define $\predim(A/B) = \predim(A) - \predim(B)$. Extend this definition to an infinite $\set{S}$-structure $A$ and a substructure $B$ by defining $\predim(A/B) = \inf\setcolon{\predim(X/X\cap B)}{X\subseteq A, |X|<\infty}$. Write $B\strong A$ if for every $B\subseteq X\subseteq A$, it is the case that $\predim(X/B)\geq 0$.

Let $A$ be some $\set{S}$-structure such that $\set{a}\strong A$ for every $a\in A$. For every finite substructure $X\subseteq A$, define $\dm_{\geometry(A)}(X) = \inf\setarg{\predim(Y)}{X\subseteq Y\subseteq A}$. The function $\dm_{\geometry(A)}$ uniquely determines the dimension function of a flat geometry $\geometry(A)$ with the same universe as $A$. We call $\geometry(A)$ the \emph{geometry associated with} $A$.

\subsection{The ab initio clique and symmetric constructions}

The following definitions and facts are taken from \cite{HassonMermelstein2017}.

\begin{definition}
Define $\Cclq_0$ to be the class of $\set{S}$-structures $A$ such that whenever $K_1,K_2\in\maxcliques{A}$ are distinct, then $|K_1\cap K_2| < n$.

Define $\Cclq$ to be the class of $\set{S}$-structures $A\in\Cclq_0$ such that $\set{a}\strong A$ for every $a\in A$.

Define $\Csym$ to be the class of $\set{S}$-structures $A\in\Cclq$ such that $\maxcliques{A}\subseteq [A]^n$.
\end{definition}

\begin{fact}\*
\begin{itemize}
\item
The function $\predim:\Cclq_0\to \Ints$ is submodular. That is, letting $D\in\Cclq_0$ and letting $A,B,A\cup B,A\cap B\subseteq D$ be induced substructures, we have
\[
\predim(A\cup B) + \predim(A\cap B) \leq \predim(A) + \predim(B).
\]
\item
The relation $\strong$ is transitive for structures in $\Cclq_0$.
\end{itemize}
\end{fact}

\begin{definition}
Let $A_1,A_2\in \Cclq_0$ and let $B=A_1\cap A_2$ be a common induced substructure. Define the \emph{standard amalgam} of $A_1$ and $A_2$ over $B$ to be the unique $\set{S}$-structure $D$ whose universe is $A_1\cup A_2$ such that $\maxcliques{D} = M\cup M'$ where
\begin{gather*}
M = \setcolon{K\in\maxcliques{A_1}\cup \maxcliques{A_2}}{|K\cap B|<n}
\\
M' = \setcolon{K_1\cup K_2}{K_1\in\maxcliques{A_1}, K_2\in\maxcliques{A_2}, |K_1\cap K_2| \geq n}
\end{gather*}
\end{definition}

\begin{fact}
Let $A_1,A_2\in \Cclq_0$ be such that $B = A_1\cap A_2$ is a common substructure. Let $D$ be the standard amalgam of $A_1$ and $A_2$ over $B$. Then $\predim(D/A_1) = \predim(A_2/B)$. In particular, if $A_1,A_2\in\Cclq$, then $D\in\Cclq$.
\end{fact}

For $A,B\in \Cclq_0$, say that a first order embedding $f:A\to B$ is \emph{strong} if $f[A]\strong B$. The classes $\Cclq, \Csym$ are clearly closed under isomorphisms and substructures and have JEP. By the above fact, they also have AP with respect to strong embeddings. Thus, they each have a unique (up to isomorphism) countable generic structure. We denote these generic structures $\Mclq$ and $\Msym$ respectively. $\Msym$ is known as the $n$-ary uncollapsed symmetric ab initio Hrushovski construction.

\begin{fact}\label{symgeo is clqgeo}
$\geometry(\Mclq)\cong \geometry(\Msym)$. Moreover, for every $A^{\clq}\in\Cclq$ there is some $A^{\sym}\in\Csym$ such that $\geometry(A^{\clq}) = \geometry(A^{\sym})$, and vice versa.
\end{fact}

\section{The geometric construction} For a geometry $\pg{G}$ on a set $G$, observe that $\langle G, D_n\rangle$ is an $\set{S}$-structure in $\Cclq_0$. Say that an $\set{S}$-structure $A$ is \emph{geometric} if
\begin{itemize}
\item
Whenever $X\subseteq A$ with $|X|\geq n$ and $\predim(X) < n$, then there exists a unique $K\in\maxcliques{A}$ with $X\subseteq K$.
\end{itemize}

\begin{observation}
\label{largeDependentIsClique}
If $A$ is a geometric $\set{S}$-structure and $X\subseteq A$ is such that $|X|\geq n$ and $\predim(X)< n$, then $X$ is a clique and $\predim(X) = n-1$. Thus, $X\strong A$ for any $X\subseteq A$ with $|X|<n$.
\end{observation}

\begin{definition}
Define $\Cgeo$ to be the class of geometric $\set{S}$-structures.
\end{definition}

\begin{observation}
The class $\Cgeo$ is a subclass of $\Cclq$.
\end{observation}

\begin{definition}
Let $A_1,A_2\in \Cgeo$ and let $B=A_1\cap A_2$ be a common induced substructure. Define the \emph{geometric amalgam} of $A_1$ and $A_2$ over $B$ to be the unique $\set{S}$-structure $D$ whose universe is $A_1\cup A_2$ such that $\maxcliques{D} = M\cup M_1\cup M_2$ where
\begin{gather*}
M = \setcolon{K_1\cup K_2}{K_1\in\maxcliques{A_1}, K_2\in\maxcliques{A_2}, |K_1\cap K_2| \geq n-1}
\\
M_1 = \setcolon{K\in\maxcliques{A_1}}{\forall L\in\maxcliques{A_2} ~~ |K\cap L|<n-1}
\\
M_2 = \setcolon{K\in\maxcliques{A_2}}{\forall L\in\maxcliques{A_1} ~~ |K\cap L|<n-1}
\end{gather*}
\end{definition}
The above amalgam differs from the standard amalgam in the definition of $M$, where we take the union of cliques if they overlap at $n-1$ points rather than in $n$ points.
\begin{lemma}
Let $A_1,A_2\in \Cgeo$ with $B=A_1\cap A_2$ and $B\strong A_1$. Let $D$ be the geometric amalgam of $A_1$ and $A_2$ over $B$. Then $D\in\Cgeo$.
\end{lemma}

\begin{proof}
Let $X\subseteq D$ with $|X|\geq n$ and $\predim(X)< n$. We must show there is a unique $K\in\maxcliques{D}$ such that $X\subseteq K$. Since $A_1,A_2\in\Cgeo$, we may assume that $X\nsubseteq A_1$ and $X\nsubseteq A_2$.

We show that there is at most one clique containing $X$. Assume there are $K,L\in\maxcliques{D}$ distinct such that $X\subseteq K\cap L$. Then $K = K_1\cup K_2$, $L = L_1\cup L_2$ where $K_i,L_i\in\maxcliques{A_i}$ and it must be that $|K_1\cap B|,|L_1\cap B|\geq n-1$. Let $C= (K_1\cup L_1)\cap B$, let $x\in X\setminus A_2$. Then $\predim(C\cup\set{x}/C) <0$ in contradiction to $B\strong A_1$.

We have left to show that $X$ is contained in a clique. It suffices to show that this holds of some superset of $X$, so we may enlarge $X$. Thus, assume that for any $K\in\maxcliques{D}$, if $|K\cap X|\geq n$, then $K\subseteq X$. This implies $\predim(X/X\cap A_2) = \predim(X\cap A_1/X\cap B)$.

We claim $|X\cap A_1|\geq n$. Assume $|X\cap A_1| < n$. Then $\predim(X/X\cap A_2) = |X\cap (A_1\setminus A_2)|$. It must be that $|X\cap A_2| \geq n$, for otherwise $\predim(X) = |X| \geq n$, in contradiction. As $\predim(X\cap A_2)< \predim(X)< n$ and $|X\cap A_2|\geq n$, this is a contradiction to $A_2$ being geometric. A symmetric argument yields that also $|X\cap A_2| \geq n$.

Now, since $B\strong A_1$, we have that $\predim(X/X\cap A_2) \geq 0$. Thus, $\predim(X\cap A_2) \leq \predim(X) < n$. Since $|X\cap A_2| \geq n$, by Observation \ref{largeDependentIsClique}, there exists $K_2\in\maxcliques{A_2}$ such that $X\cap A_2= K_2$ and $\predim(X\cap A_2) = n-1$. As $\predim(X\cap A_1/X\cap B) =\predim(X/X\cap A_2) = 0$ and, since $X\cap B\subseteq K_2$, also $\predim(X\cap B) \leq n-1$, we have $\predim(X\cap A_1) <n$. Therefore, there is some clique $K_1\in\maxcliques{A_1}$ such that $X\cap A_1 = K_1$. If $|K_1\cap K_2|<n-1$, then $\predim(X) \geq n$ in contradiction. So $|K_1\cap K_2|\geq n-1$ and $K_1\cup K_2\in\maxcliques{D}$ contains $X$.
\end{proof}

\begin{remark}
The geometric amalgam is also an amalgam for $\Cclq$.
\end{remark}

Clearly $\Cgeo$ is closed under isomorphism and substructures, and has JEP. The above lemma gives us AP for $\Cgeo$, and so $\Cgeo$ has a countable generic limit $\Mgeo$. The closed sets of dimension $(n-1)$ in $\Mgeo$ are exactly $\maxcliques{\Mgeo}$.

\section{$\Mgeo$ as the geometry of a generic structure}
\begin{definition}
Let $\pg{A}$ be a flat, $(n-1)$ pure geometry with underlying set $A$. We define $\pg{A}^{\geo}$ to be the $\set{S}$-structure with underlying set $A$ and $\maxcliques{\pg{A}^{\geo}} = \setarg{\cl_{\pg{A}}(B)}{B\in [A]^{n-1}, \cl_{\pg{A}}(B) \neq B}$.
\end{definition}

\begin{observation}
If $\pg{B}\subseteq\pg{A}$, then $\pg{B}^{\geo}\subseteq\pg{A}^{\geo}$.
\end{observation}

\begin{observation}\label{geoOperatorIsIdempotentOnGeo}
For every $A\in \Cgeo$,
\begin{itemize}
\item
$\geometry(A)$ is $(n-1)$-pure.
\item
$(\geometry(A))^{\geo} = A$.
\end{itemize}
\end{observation}

\begin{lemma}\label{deltaGreaterThand}
For any flat, $(n-1)$-pure geometry $\pg{B}$ and $\pg{A}\subseteq \pg{B}$, we have $\predim(\pg{A}^{\geo})\geq \dm_{\pg{B}}(A)$.
\end{lemma}

\begin{proof}
Assume $|A|>n-1$, for otherwise $\predim(\pg{A}^{\geo})= \dm_{\pg{B}}(A) = |A|$. Let
\[
\widehat{A}_0 = A\cup\bigcup_{X\in[A]^{n-1}}\cl_{\pg{B}}(X).
\]
Clearly $\predim(\pg{A}^{\geo}) \geq \predim(\widehat{\pg{A}}^{\geo})$ and $\dm_{\pg{B}}(\widehat{A})\geq\dm_{\pg{B}}(A)$, so we may assume $\widehat{A}=A$. Letting $\check{A} = \bigcup \maxcliques{\pg{A}^{\geo}}$, note that $\predim(\pg{A}^{\geo}) -\predim(\check{\pg{A}}^{\geo}) = |A\setminus \check{A}| \geq \dm_{\pg{B}}(A) - \dm_{\pg{B}}(\check{A})$, so we may also assume $\check{A} = A$.

Enumerate $\maxcliques{\pg{A}^{\geo}}=\set{\otn{E}{k}}$ and note that $A=\bigcup_{i=1}^k E_i$. Observe that by $(n-1)$-purity, $\dm_{\pg{B}}(E_s) = |E_s|$ whenever $|s|\geq 2$. Then
\begin{align*}
\sum_{\emptyset\neq s\subseteq \set{1,\dots,k}}(-1)^{|s|+1}\dm_{\pg{B}}(E_s) &= \sum_{i=1}^k \dm_{\pg{B}}(E_i) + \sum_{\substack{s\subseteq \set{1,\dots,k}\\|s|\geq 2}} (-1)^{|s|+1}|E_s|
\\
&= \sum_{i=1}^k |E_i| -\cardstar{E_i} + \sum_{\substack{s\subseteq \set{1,\dots,k}\\|s|\geq 2}} (-1)^{|s|+1}|E_s|
\\
&= \sum_{\emptyset\neq s\subseteq \set{1,\dots,k}} (-1)^{|s|+1}|E_s| - \sum_{i=1}^k \cardstar{E_i}
\\
&= |\bigcup_{i=1}^k E_i| - \sum_{i=1}^k \cardstar{E_i}
\\
&= |A| - s(\pg{A}^{\geo}) = \predim(\pg{A}^{\geo})
\end{align*}
By flatness, $-\dm_{\pg{B}}(A) +\sum_{\emptyset\neq s\subseteq \set{1,\dots,k}}(-1)^{|S|+1}\dm_{\pg{B}}(E_s) \geq 0$, so $\predim(\pg{A}^{\geo}) \geq \dm_{\pg{B}}(A)$.
\end{proof}

\begin{lemma}
\label{Cclq (n-1)-pure is Cgeo}
Let $A\in\Cclq$ be such that $\pg{A}:=\geometry(A)$ is $(n-1)$-pure, then $\pg{A}^{\geo}\in\Cgeo$.
\end{lemma}

\begin{proof}
By the above lemma, we have that $\pg{A}^{\geo}\in\Cclq$. Denote by $\predim^A$ the restriction of $\predim$ to substructures of $A$, and by $\predim^{\pg{A}^{\geo}}$ the restriction of $\predim$ to substructures of $\pg{A}^{\geo}$.

Let $X\subseteq \pg{A}^{\geo}$ with $|X|\geq n$ and $\predim^{\pg{A}^{\geo}}(X)<n$ such that $\predim^{\pg{A}^{\geo}}(X)$ is minimal. Note that by our choice of $X$, it must be that $\bigcup \maxcliques{X} = X$. We may assume $\maxcliques{X} \subseteq\maxcliques{\pg{A}^{\geo}}$, replacing each clique with the maximal clique containing it can only lower $\predim^{\pg{A}^{\geo}}(X)$. Enumerate $\maxcliques{X} = \set{\otn{K}{r}}$. We show inductively that for every $m\leq r$, for the set $X_m = \bigcup_{i=1}^m K_i$ we have $\predim^{\pg{A}^{\geo}}(X_m)\geq \predim^A(X_m)$. Assume this holds for $m$.

\medskip
\noindent Case 1: $|K_{m+1}\cap X_m|\geq n-1$, then $\predim^{\pg{A}^{\geo}}(X_{m+1}/X_m) = 0$. By submodularity,
\[
\predim^A(X_{m+1}/X_m) \leq \predim^A(K_{m+1}/K_{m+1}\cap X_m).
\]
We know $\predim^A(K_{m+1}\cap X_m)\geq n-1$ and $\predim^A(K_{m+1}) = n-1$, so $\predim^A(X_{m+1}/X_m) \leq 0$.

\medskip
\noindent Case 2: $|K_{m+1}\cap X_m| < n-1$. Then
\begin{align*}
\predim^{\pg{A}^{\geo}}(X_{m+1}/X_m) &= |K_{m+1}\setminus X_m| - \cardstar{K_{m+1}}
\\
&= |K_{m+1}| - |K_{m+1}\cap X_m| -(|K_{m+1}| - (n-1))
\\
&= (n-1) -|K_{m+1}\cap X_m|.
\end{align*}
On the other hand, by submodularity
\begin{align*}
\predim^A(X_{m+1}/X_m) &\leq \predim^A(K_{m+1}/K_{m+1}\cap X_m)
\\
&= \predim^A(K_{m+1}) - \predim^A(K_{m+1}\cap X_m)
\\
&= (n-1) - |K_{m+1}\cap X_m|.
\end{align*}

In any case, $\predim^{\pg{A}^{\geo}}(X_{m+1}/X_m)\geq \predim^A(X_{m+1}/X_m)$ and $\predim^{\pg{A}^{\geo}}(X_m)\geq \predim^A(X_m)$, so $\predim^{\pg{A}^{\geo}}(X_{m+1})\geq \predim^A(X_{m+1})$.

As $X=X_r$ we have $\predim^A(X) \leq \predim^{\pg{A}^{\geo}}(X)<n$. So $X\subseteq\cl_{\pg{A}}(X)\in \maxcliques{\pg{A}^{\geo}}$.
\end{proof}

The following lemma is an adaptation of \cite[Second Changing Lemma]{DavidMarcoOne}, with the same proof.

\begin{lemma}\label{changingLemma}
Let $A,D\in\Csym$ with $A\strong D$. Let $B\in\Cclq$ with $\geometry(B) = \geometry(A)$. Let $D'$ be the structure with the same universe as $D$, and $S^{D'} = (S^{D}\setminus S^{A})\cup S^{B}$. Then $B\strong D'$ and $\geometry(D') = \geometry(D)$.\qed
\end{lemma}

\begin{lemma}
\label{mixedAmalgam}
Let $A\in\Cclq$ be such that $\pg{A}:=\geometry(A)$ is $(n-1)$-pure. Let $\widehat{B}\in \Cgeo$ be such that $\pg{A}^{\geo}\strong \widehat{B}$. Then there exists $B\in \Cclq$ such that $A\strong B$, $\geometry(B)$ is $(n-1)$-pure, and $(\geometry(B))^{\geo} = \widehat{B}$.
\end{lemma}

\begin{proof}
By Fact \ref{symgeo is clqgeo}, let $A^{\sym}$ be such that $\geometry(A^{\sym})=\pg{A}$. Again by Fact \ref{symgeo is clqgeo}, as $\widehat{B}\in\Cclq$, let $B^{\sym}\in\Csym$ be such that $\geometry(B^{\sym}) = \geometry(\widehat{B})$. Let $B$ be the structure obtained from $B^{\sym}$ by replacing $A^{\sym}$ with $A$, as in Lemma \ref{changingLemma}. Then $A\strong B$, $B\in\Cclq$, and $\geometry(B) = \geometry(B^{\sym})=\geometry(\widehat{B})$. By Observation \ref{geoOperatorIsIdempotentOnGeo}, $\geometry(B)$ is $(n-1)$-pure and $(\geometry(B))^{\geo} = \widehat{B}$.
\end{proof}

Let $\mathbb{C}$ be the class of structures $A\in \Csym$ such that $\geometry(A)$ is $(n-1)$-pure. This is an amalgamation class closed under the standard amalgam (which is the free amalgam, in this case). Denote by $\mathbb{M}$ the generic structure of $\mathbb{C}$. For every $A\in\mathbb{C}$, define $\widehat{A} = (\geometry(A))^{\geo}$.

The following is a special case of Lemma 15 of \cite{Hns}.

\begin{lemma}
\label{predimension of closed is dimension}
Let $A\in\mathbb{C}$ and let $\pg{A}:=\geometry(A)$. Then $\predim(\widehat{A}) = \dm_{\pg{A}}(A)$.
\end{lemma}

\begin{proof}
Enumerate $\maxcliques{\widehat{A}} = \set{\otn{E}{k}}$. Note that $S(A) = \bigcup_{i=1}^k S(E_i)$, because whenever $(\otn{a}{n})\in S(A)$, then $\cl(\set{\otn{x}{n}})\in\maxcliques{\widehat{A}}$. Then $\dm_{\pg{A}}(A) - \dm_{\pg{A}}(E_{\emptyset}) = |A\setminus E_{\emptyset}| = \predim(\widehat{A}) -\predim(E_{\emptyset})$, and so we may assume, similarly to as in the proof of \ref{deltaGreaterThand}, that $\widehat{A} = E_{\emptyset}$.

Observe that $E_s\strong A$ for any $s\neq \emptyset$, as an intersection of closed sets. Recall that in the proof of lemma \ref{deltaGreaterThand} we saw that
\[
\predim(\widehat{A}) = \sum_{\emptyset\neq s\subseteq \set{1,\dots,k}}(-1)^{|S|+1}\dm_{\pg{A}}(E_s).
\]
Then
\begin{align*}
\predim(\widehat{A}) &= \sum_{\emptyset\neq s\subseteq \set{1,\dots,k}} (-1)^{|s|+1}\dm_{\pg{A}}(E_s)
\\
&= \sum_{\emptyset\neq s\subseteq \set{1,\dots,k}} (-1)^{|s|+1}\predim(E_s)
\\
&= \sum_{\emptyset\neq s\subseteq \set{1,\dots,k}} (-1)^{|s|+1}|E_s| - \sum_{\emptyset\neq s\subseteq \set{1,\dots,k}} (-1)^{|s|+1}\frac{|S(E_s)|}{n!}
\\
&= |A| - \frac{|S(A)|}{n!} = \predim(A) = \dm_{\pg{A}}(A)
\end{align*}
where the equality between the third and fourth line is by the inclusion-exclusion principle.
\end{proof}

\begin{prop}
$\widehat{\mathbb{M}} \cong \Mgeo$
\end{prop}

\begin{proof}
By generalizing the method of \cite[3.6.7]{Mermelstein2016}, we need only show that:
\begin{itemize}
\item
If $A\in \mathbb{C}$, then $\widehat{A}\in \Cgeo$.
\item
If $A\strong B\in\mathbb{C}$, then $\widehat{A}\subseteq\widehat{B}$.
\item
If $A\in \mathbb{C}$, $\widehat{B}\in\Cgeo$ such that $\widehat{A}\strong\widehat{B}$, then there exists $A\strong C\in\mathbb{C}$ such that $\widehat{B}\strong\widehat{C}$.
\end{itemize}

The first point is Lemma \ref{Cclq (n-1)-pure is Cgeo}, the second point is because $A\strong B$ implies $\geometry(A)\subseteq\geometry(B)$, the third point follows from Lemma \ref{mixedAmalgam}.
\end{proof}

\begin{lemma}\label{strongInMstrongInMgeo}
Whenever $A\strong \mathbb{M}$, then $\widehat{A}\strong\widehat{M}$. Additionally, $\predim(A) = \predim(\widehat{A})$.
\end{lemma}

\begin{proof}
For the first part, apply \cite[Lemma 3.6.5]{Mermelstein2016}. For the additional part, observe that $\geometry(A)\subseteq \geometry(\widehat{M})$. Then by Lemma \ref{predimension of closed is dimension}, we have
\begin{align*}
\predim(\widehat{A}) &= \dm_{\geometry(A)}(A)
\\
&= \dm_{\geometry(\mathbb{M})}(A)
\\
&= \predim(A).
\end{align*}
\end{proof}

Note that the following theorem is about \emph{equality} of pregeometries, and not mere isomorphism.
\begin{theorem}
Identifying $\widehat{\mathbb{M}}$ with $\Mgeo$, we have $\geometry(\mathbb{M})=\geometry(\Mgeo)$.
\end{theorem}

\begin{proof}
Let $A\subseteq\mathbb{M}$ be finite.

Let $B\strong \mathbb{M}$ be the self sufficient closure of $A$ in $\mathbb{M}$. Then
\[
\dm_{\geometry(\mathbb{M})}(A)= \predim(B) = \predim(\widehat{B})\geq \dm_{\geometry(\Mgeo)}(A).
\]
Let $\widehat{D}\strong \Mgeo$ be the self sufficient closure of $\widehat{A}$ in $\Mgeo$. Let $D\subseteq\mathbb{M}$ be the structure induced by $\mathbb{M}$ on the underlying set of $\widehat{D}$.
By Lemma \ref{deltaGreaterThand} we have 
\[
\dm_{\geometry(\mathbb{M})}(A)\leq \dm_{\geometry(\mathbb{M})}(D)\leq \predim(\widehat{D})=\dm_{\geometry(\Mgeo)}(A).
\]
\end{proof}

The above theorem shows that $\Mgeo$ is infact the reduct of $\geometry(\mathbb{M})=\geometry(\Mgeo)$ to $D_n$. So clearly, the structure of $\Mgeo$ can be read off the geometry of $\Mgeo$ and vice versa.

\bibliographystyle{alpha}
\bibliography{../myrefs}
\end{document}